\documentclass[graybox]{svmult}

\usepackage{mathptmx}       
\usepackage{helvet}         
\usepackage{courier}        
\usepackage{type1cm}        
\usepackage{makeidx}         
\usepackage{graphicx}        
\usepackage{multicol}        
\usepackage[bottom]{footmisc}
\usepackage{amssymb, url}
\usepackage{amsmath,color}
\usepackage{epsfig}
\usepackage{algorithm}
\usepackage{algorithmic}

\def\jdlqed{\vbox{\hrule \hbox{\vrule\hbox to
5pt{\vbox to 6pt{\vfil}\hfil}\vrule}\hrule}}

\newcommand{\real}{\mathbb R}
\newcommand{\R}{\mathbb R}
\newcommand{\Q}{\mathbb Q}
\newcommand{\Z}{\mathbb Z}
\newcommand{\1}{\textbf 1}

\newcommand{\sg}{\mbox{\rm Sg}}

\newcommand{\frob}{\mbox{\rm F}}
\newcommand{\dfrob}{\mbox{\rm g}}
\newcommand{\interior}{\mbox{\rm int}}

\renewcommand{\intercal}{T}

\DeclareMathOperator{\sat}{sat}
\DeclareMathOperator{\cone}{cone}

\DeclareMathOperator{\lattice}{lattice}

\makeindex             

\begin{document}

\title*{Parametric Polyhedra with at least $k$ Lattice Points: \\ Their Semigroup Structure  and the $k$-Frobenius Problem}

\author{Iskander Aliev, Jes\'us A. De Loera, and Quentin Louveaux}
\institute{Cardiff University, UK,\\
\email{AlievI@cardiff.ac.uk},\\
\and
University of California, Davis\\
\email{deloera@math.ucdavis.edu}\\
\and
Universit\'e de Li\`ege, Belgium \\
\email{q.louveaux@ulg.ac.be}}

%
%
\maketitle

\abstract{Given an integral $d \times n$ matrix $A$, the well-studied affine semigroup $\sg (A)=\{ b : Ax=b, \ x \in \Z^n, x \geq 0\}$ can be stratified by the number of lattice points inside the parametric
 polyhedra $P_A(b)=\{x: Ax=b, x\geq0\}$. Such families of parametric polyhedra appear in many areas of combinatorics, convex geometry, algebra and number theory.
The key themes of this paper are: (1) A structure theory that characterizes precisely the subset $\sg_{\geq k}(A)$ of all vectors $b \in \sg(A)$ such that
$P_A(b) \cap \Z^n $  has \emph{at least} $k$ solutions. We demonstrate that this set is finitely generated, it is a union of translated copies of a semigroup
which can be computed explicitly via Hilbert bases computations. Related results can be derived for those right-hand-side vectors $b$ for which $P_A(b) \cap \Z^n$
has \emph{exactly} $k$ solutions or \emph{fewer than} $k$ solutions.
(2) A computational complexity theory. We  show that, when $n$, $k$ are fixed natural numbers, one can compute in polynomial time an
encoding of $\sg_{\geq k}(A)$ as a multivariate generating function, using a short sum of rational functions. As a consequence,
one can identify all right-hand-side vectors of bounded norm that have at least $k$ solutions.
(3) Applications and computation for the $k$-Frobenius numbers. Using Generating functions we prove that for fixed $n,k$ the
$k$-Frobenius number can be computed in polynomial time. This generalizes a well-known result for $k=1$ by R. Kannan.
Using some adaptation of dynamic programming we show some practical computations of $k$-Frobenius numbers and their relatives.
}

\keywordname{\;Lattice points in polyhedra, linear Diophantine equations, Hilbert bases, combinatorial commutative algebra, combinatorial number theory, affine semigroups, Frobenius numbers}

\subclassname{\;52C07; 52B; 11D07; 05E40; 05A15}

\section{Introduction}

An \emph{ affine semigroup} is a semigroup (always containing a zero element) which is finitely generated and can be embedded in $\Z^n$ for some $n$. This paper studies special polyhedral affine semigroups that appear naturally in many interesting problems in combinatorics, convexity, commutative algebra, and number theory and that can be described in very explicit terms.

Given an integer matrix $A \in \Z^{d\times n}$ and a vector $b \in \Z^d$, we study the semigroup $\sg (A)=\{ b : b=Ax, \ x \in \Z^n, x \geq 0\}$. Geometrically $\sg(A)$ can be described as \emph{some} of the lattice points inside the convex polyhedral cone $\cone (A)$ of non-negative linear combinations of the columns of $A$. It is well-known that  $\sg(A) \subset \cone (A) \cap \Z^d$, but the equality is not always true. It is also well-known that $\cone (A) \cap \Z^d$ is  a finitely generated semigroup, this time with generators given by the \emph{Hilbert bases} of $\cone (A)$ \cite{henkweismantel,schrijver}.   The study of the difference between $\sg(A)$ and $\cone (A) \cap \Z^n$ is quite interesting (e.g., it has been part of many papers in commutative algebra about semigroups and their rings. See  \cite{barucci,nill,stanley0,sturmfels} and the references therein).

Practically speaking, membership of  $b$ in the semigroup $\sg(A)$ reduces to the challenge, given a vector $b$, to find whether the linear Diophantine system
\begin{equation}\label{feas1}
A x = b, \quad x \geq 0, \quad x \in \Z^n\,,
\end{equation}
has a solution or not. Geometrically, problem  (\ref{feas1})  asks whether there is at least one lattice point inside the \emph{parametric polyhedron} $P_A(b)=\{x : Ax=b, x\geq 0\}$.

Now, for a given integer $k$,  there are three variations of the classical feasibility problem above
that in a natural way measure the number of integer points in $P_A(b)$:

\begin{itemize}
\item Are there \emph{at least } $k$ distinct solutions for $P_A(b) \cap \Z^n$? If yes, we say that the polyhedron $P_A(b)$ is
 \emph{$ \geq k$-feasible}.
\item Are there  \emph{exactly} $k$ distinct solutions for $P_A(b) \cap \Z^n$? If yes, we say that the polyhedron $P_A(b)$ is
 \emph{$=k$-feasible}.
\item Are there \emph{less than} $k$ distinct solutions for $P_A(b) \cap \Z^n$? If yes, we say that the polyhedron $P_A(b)$ is
 \emph{$<k$-feasible}.
\end{itemize}

The standard integer feasibility problem is just the problem of deciding whether $P_A(b)$ is $\geq 1$-feasible.
We call these three problems, \emph{the fundamental problems of $k$-feasibility}, thus these problems
are already NP-hard in complexity for $k=1$. For convenience $IP_A(b)$ will denote the integer points inside $P_A(b)$.
Similarly, we say that $b$ is $\geq k$-feasible, or respectively, $=k$-feasible or  $<k$-feasible if the corresponding polyhedron $P_A(b)$ is.

This paper investigates the question of determining, given an integral matrix $A$, which right-hand-side vectors $b$ are $\geq k$-feasible, $=k$-feasible, or $ <k$-feasible and the structure of the corresponding sets of all such vectors. The first part of this paper is about a decomposition or stratification of the semigroup $\sg(A)$ by values of $k$. The decomposition structure we unveil is as translated semigroups that group together all elements $b \in \sg(A)$ that are $\geq k$-feasible (similarly for the other cases).  In the second part we consider the computational complexity of deciding whether a vector $b$ is $\geq k$-feasible (and similarly for $<k$ and $=k$ feasibility). The final answer is a clean nice application of the theory of combinatorial commutative algebra.
The algorithms we propose in the second part are rather interesting theoretically, yet not practical beyond a few dozen variables, thus the third part focuses in
computer experimentation  with  the $k$-Frobenius number, an invariant of interest in combinatorial number theory which is a special case for $1 \times n$
matrices.

\subsection*{Motivation, prior and related work}

Why study $k$-feasibility? The three $k$-feasibility problems appear directly in a wide range of situations and are also special cases of important algebraic and geometric problems we describe now.


To begin, $k$-feasibility is strongly connected to a classical problem in combinatorial number theory:
Let $a$ be a positive integral $n$-dimensional primitive vector,
i.e., $a=(a_1,\dots,a_n)^T\in\Z_{>0}^n$ with $\gcd(a_1,\dots,a_n)=1$. For a positive integer $k$, the
{\em $k$-Frobenius number} $\frob_k(a)$ is the largest number which cannot be
represented in \emph{at least} $k$ different ways as a non-negative integral combination of the $a_i$'s.  Thus, putting $A=a^T$,
\begin{equation}\label{Knapsack_Frobenius}
  \frob_k(a)=\max\{b\in\Z : IP_A(b)\mbox{ is }<k\mbox{-feasible}\}.
\end{equation}

When $k=1$ this was studied by a large number of authors and both the structure and algorithmic properties are well-understood.
Computing $\frob_1(a)$ when $n$ is not fixed is an NP-hard problem (Ramirez Alfonsin \cite{Alf1}). On the other hand, for any fixed $n$
the classical Frobenius number can be found in polynomial time by sophisticated deep algorithms due to Kannan \cite{Kannan} and later
Barvinok and Woods \cite{newbar}. The general problem of finding $\frob_1(a)$ has been traditionally referred to as the {\em Frobenius problem}.
There is a rich literature on the various aspects of this question. For a comprehensive and extensive
survey  we refer the reader to the book of Ramirez Alfonsin \cite{frob}. 

The $k$-feasibility framework yields a generalization of the Frobenius number which was introduced and studied
by Beck and Robins in \cite{BeckRobins:extension}. Now one is interested on finding the largest number $b$ that
cannot be represented in more than $k$ ways. Beck and Robins gave formulas for $n=2$ of the $k$-Frobenius number,
but for general $n$ and $k$ only bounds on the $k$-Frobenius number $\frob_k(a)$ are available
(see \cite{AFH},\cite{AHL} and \cite{FS} for prior work). Since then, many number theorists
have contributed to the topic (e.g., see \cite{brownetal,shallit2011} and the references there). Note also that
Frobenius problem is highly related to what optimizers call the Knapsack problem and it has many applications 
(see \cite{aardaletal3,AHL,AOW,kellereretalbook} and references therein).


The structure of $k$-feasibility is also closely related to the properties of affine semigroups and the associated commutative rings.
Fix a semigroup $S$ and a field $K$. The \emph{semigroup algebra} over $S$ with coefficients in $K$, denoted $K[S]$, is the set of  finite formal sums of elements of the form $cx^q$ for $q\in S$ and $c \in K$. Here the multiplication is given by the structure of
the semigroup addition $x^a \cdot x^b= x^{a+b}$.  The ring theoretic properties of $K[S]$ and the structure of $S$ are interconnected, e.g.,
the Krull dimension of $K[S]$ equals the rank of $S$. There are a number of questions about semigroup rings which
play an important role in the theory of toric varieties (as the coordinate rings of general toric varieties) and the structure of lattice polyhedra,
rational cones and linear Diophantine systems (see e.g., \cite{haaseetal,nill,sturmfels}
and the references therein). For example, the question of when is  the semigroup algebra normal, is equivalent to the case of
determining $1$-feasibility (see e.g., \cite{brunskoch}).  There are many interesting affine semigroups in algebraic combinatorics, e.g., through the work of several authors,
 the values of Littlewood-Richardson coefficients, or more general Clebsch-Gordan coefficients, can be interpreted (using the
lattice point interpretation of those numbers) as the values of $IP_A(b)$ (see \cite{KT99,JDTM,pakvallejo}  and their references).

Given a finite list of integers $S =[a_1, \dots, a_n]$ the numerical semigroup $\langle S\rangle$ with minimal generating set $S$ is the set of of all linear non-negative combinations. For a fixed element $b \in \langle S \rangle$  a \emph{factorization} of $b$ is an expression $b = \lambda_1 a_1 +\dots + \lambda_n a_n$. The integer vector $\bar{\lambda}=(\lambda_1, \dots, \lambda_n)$ uniquely defines the factorization. The \emph{factorization set} of $b$ is $Z(b)$ is the set of all its factorizations of $b$ (note this is exactly what we call $IP_A(b) $ for $A=[a_1 a_2 \dots a_n]$). Algebraists  study the factorization properties in monoids and integral domains. Numerical monoids provide an excellent venue to explore various measurements of non-unique factorization (see \cite{BGT,semig1,semig2,semig3,semibook} and references there). An interesting connection to $k$-feasibility is that the cardinality of $Z(b)$ is the multivariate Hilbert function of certain graded rings \cite{chriso}. The gaps
of semigroups (see e.g., \cite{gapsramirezalfonsin})  correspond to the $<1$-feasible elements. Here we describe the rest of the stratification of the elements in the semigroup.


Naturally  $k$-feasibility problems  have also interesting applications in combinatorics and statistics:  A \emph{vector partition function} of the $d\times n$ matrix $A$ is
 $$ \phi_A(b) \quad = \quad
\# \, \bigl\{\,x \,:\, Ax=b, x \geq 0 , \,\, x \, \,{\rm integral} \,
\bigr\}. $$
It is well-known \cite{sturmfelsvpf} the function $\phi_A$ is piecewise polynomial of degree $n-rank(A)$.
Its domains of polynomiality are convex polyhedral cones, these are the so-called
{\em chambers} of $A$ \cite{triangbook}. Vector partition functions for special matrices are crucial in several topics of
mathematics (see \cite{baldonietal,brionvergne,deconcinietal, schmidtbincer,sturmfelsvpf,szenesvergne} and the references there). In the particular case that $b$ is fixed and one considers dilations of $b$, one has an Ehrhart function
enumerating the lattice points on a dilated polyhedron dilations $P_A(Nb)$ as $N$ grows to infinity. The Ehrhart functions of parametric polytopes of the form $P_A(b)$ has been studied in depth by many researchers (see \cite{barvinokzurichbook,ehrhartbook} and references therein. See also the article \cite{braun} in this volume).
Our investigations aim to describe those vectors $b$ for which the number of lattice points has at least (or at most or exactly) a value $k$ and uncover their structure.

Consider as concrete example  the widely popular recreational puzzle \emph{sudoku}, each instance can be thought of as an integer linear
program where the hints provided in some of the entries are the given right-hand-sides
of the problem. Of course in that case newspapers  wish to give readers a puzzle where the solution is unique ($k=1$), but several authors
have studied recently conditions on which there are $k$ possible solutions.
It is not difficult to see that this is a special case of a {\em 3-dimensional  transportation problem} that is, the question to decide whether
the set of \emph{integer} feasible solutions of the $r\times s\times t$-transportation problem
\[
\left\{x\in\Z^{rst}:\sum_{i=1}^r x_{ijk}=u_{jk},\sum_{j=1}^s x_{ijk}=v_{ik},\sum_{k=1}^t
x_{ijk}=w_{ij}, x_{ijk}\geq 0\right\}
\]
has a unique solution given right-hand sides $u,v,w$. Another application of $k$-feasibility appears in statistics, concretely in
application in the data security problem of  {\em multi-way contingency tables}, because when the number
of solutions is small, e.g. unique, the margins of the statistical table may disclose personal information which
is illegal \cite{dobra-karr-sanil2003}. The study of $k$-feasibility appears in explicitly in \cite{hemmecke}.

%

Besides the directly related problems above, $k$-feasibility questions are special cases (for parametric polyhedra of the form $\{x: Ax=b, x \geq 0\}$) of
the problem to classify polyhedra with $k$ lattice points.
Polyhedra, specially lattice polyhedra, with fixed number of (interior) lattice points play a role in many areas of pure mathematics including representation theory, algebraic geometry and combinatorial geometry. One of the challenges has been to classify them. For this purpose there has been a lot of work, going back to classical results of Minkowski and van der Corput,
 to show that the volume of a lattice polytope $P$ with $k={\rm card}({\Z}^n \cap{\rm int}\, P)\geq1$ is bounded above by a constant
that only depends in $n$ and $k$ (see e.g., \cite{lagariasziegler,pikhurko} and references therein). Similarly, the supremum of the possible number of points of ${\Z}^n$ in a  lattice polytope in $\real^n$ containing precisely $k$ points of ${\Z}^n$ in its interior, can be bounded by a constant that only depends
in $n$ and $k$. For each positive integer $k$, there exist only a finite number of  non-equivalent lattice polytopes with $k$ interior points.
There have been efforts to create a census of polytopes with $k$ interior points. For the plane this was done in \cite{rab_onelattice}
Wei and Din for $k=2$ \cite{wei_ding_two}.  W. Castryk  \cite{castryck} recently provided a census for lattice polygons with given number of interior points for $k\leq 30$. A convex body is \emph{hollow} if $k=0$, i.e., it does not contain any lattice points in its interior. The classification of hollow lattice polytopes is a much more difficult problem, but there are some results toward a classification too \cite{averkov2011maximal,nill+ziegler}. The study of the asymptotic bounds of the number of non-equivalent polyhedra with given number of interior lattice points was initiated by V. A Arnol'd \cite{vladimir} and continued by many mathematicians  \cite{barany+vershik,liu+zong}. 


\subsection*{Our Results}

This paper has six new contributions to the study of $k$-feasibility for the semigroup $\sg(A)$, 
and the associated polyhedral geometry. Throughout the paper we assume that the cone $\cone(A)$ is pointed.
The six contributions guide the structure of this paper:

\begin{enumerate}

\item First, we prove a structural result that implies that the set $\sg_{\geq k}(A)$ of $b$'s, inside the semigroup $\sg (A)$, that provide $\geq k$-feasible polytopes $P_A(b)$ is finitely generated.

Let  $\sg_{\ge k}(A)$ (respectively $\sg_{= k}(A)$ and $\sg_{< k}(A)$) be the set of right-hand side vectors $b\in \cone(A)\cap\Z^d$ that make $P_A(b)$ $\geq k$-feasible (respectively $= k$-feasible, $<k$-feasible). Note that  $\sg_{\ge 1}(A)$ is equal to $\sg(A)$, the semigroup generated by the column vectors of the matrix $A$.

Our first structural theorem  gives an algebraic description of the sets $\sg_{\ge k}(A)$ and $\sg_{<k}(A)$. Let $e_1,\ldots, e_n$
be the standard basis vectors in $\Z^n_{\ge 0}$. Following the book \cite{coxlittleoshea} we define the {\em coordinate subspace of} $\Z^n_{\ge 0}$ of dimension $r\ge 1$ determined by $e_{i_1}, \ldots, e_{i_r}$ with $i_1<\cdots<i_r$ as the set $\{e_{i_1} z_1+\cdots+ e_{i_r} z_r: z_j \in \Z_{\ge 0}\mbox{ for }1\le j\le r\}$. By the $0$-dimensional coordinate subspace of $\Z^n_{\ge 0}$ we understand the origin $0\in \Z^n_{\ge 0}$.

\begin{theorem} \label{main2}
\begin{itemize}
\item[(i)] There exists a monomial ideal $I^k(A)\subset {\mathbb Q}[x_1,\ldots, x_n]$ such that
\begin{equation}
\sg_{\ge k}(A)= \{A\lambda: \lambda \in E^k(A)\}\,,
\label{exponents}
\end{equation}
where $E^k(A)$ is the set of exponents of monomials in $I^k(A)$.

\item[(ii)] The set $\sg_{< k}(A)$ can be written as a finite union of translates of the sets $\{A\lambda: \lambda \in S\}$, where $S$ is a coordinate subspace of $\Z^n_{\ge 0}$.

\end{itemize}
\end{theorem}

By the Gordan-Dickson lemma, the ideal $I^k(A)$ is finitely generated, thus one can conclude

\begin{corollary}
 $\sg_{\ge k}(A)$ is a finite union of translated copies of the semigroup $A\Z^n_{\ge 0}$.
\end{corollary}

The proof of Theorem \ref{main2} relies on some basic facts on lattice points when we think of them as generators of monomial ideals. The
basic tool is a characterization of the complement of a monomial ideal 
(see \cite{coxlittleoshea}).
Some of the arguments are of interest for the study of affine semigroups and toric varieties \cite{BGT,sturmfels}.


Theorem \ref{main2} extends  the earlier decomposition theorem of Hemmecke, Takemura and Yoshida \cite{hemmecke} for $k=1$.
They investigated the semigroup $\sg(A)$ and the integral vectors that are {\em not} in the semigroup, but still lie within the cone $\cone(A)$ generated
by the columns of $A$.
Those authors studied $Q_{\sat}=\cone(A)\cap\lattice(A)$, where  $\lattice(A)$ is the
lattice generated by the columns of $A$. They called $H=Q_{\sat}\setminus \sg(A)$ the set of \emph{holes} of $\sg(A)$.
The set of holes $H$ may be finite or infinite, but their main result is to give a finite description of the holes as a finitely-generated set.
Our Theorem \ref{main2} was inspired by theirs. 
The elements of the set $\sg_{<k}(A)$ can be viewed as \emph{$k$-holes} (in the context of
numeric semigroups and the Frobenius number, $1$-holes have also been called \emph{gaps}, see \cite{gapsramirezalfonsin}), namely those right hand-sides $b$ for which $Ax=b$ has \emph{less than} $k$ non-negative integral solutions.
Section \ref{sectionMonomialIdeal} gives a proof of Theorem \ref{main2} that relies on basic commutative algebra.

\item Second, although traditionally the Frobenius problem has been studied for $1 \times n$ matrices, in Sections \ref{Asymptotic}-\ref{proof_ideal_bounds}
we discuss $\frob_k(A)$, a generalization of $k$-Frobenius number, but this time applicable to all matrices. We explain
the meaning of this generalized $k$-Frobenius number to the structure of the set $\sg_{\ge k}(A)$ when seen far away from the origin moving toward
asymptotic directions inside the $\cone(A)$.  Essentially, under some natural assumptions for the matrix $A$, the set $\sg_{\ge k}(A)$ can be decomposed into the set of all integer points in the interior of a
certain translated cone and a smaller complex complementary set. We discuss the location of such a cone along a given direction $c$ in the interior of $\cone(A)$. The key goal of Sections \ref{Asymptotic}-\ref{proof_ideal_bounds} is to derive the lower and upper bounds for $\frob_k(A)$.
 Based on the results obtained in \cite{AH}, \cite{AHL} 
we show that $\frob_k(A)$ is bounded from above in terms of $\det(AA^T)$ and $k$.
\begin{theorem} Let $A$ be a matrix in $\Z^{d\times n}$, $1\leq d<n$, satisfying
\begin{equation}
\begin{split}
{\rm i)}&\,\, \gcd\left(\det(A_{I_d}) : A_{I_d}\text{ is an $d\times
    d$ minor of }A\right)=1, \\
{\rm ii)}&\,\, \{x\in\R^n_{\ge 0}: A\,x=0\}=\{0\}\,.
\end{split}
\label{assumption}
\end{equation}
Then the $k$-Frobenius number associated with  $A$ satisfies the inequality
\begin{equation}
\frob_k(A)\le \frac{n-d}{2(n-d+1)^{1/2}} \det(AA^T)+ \frac{(k-1)^{1/(n-d)}}{2(n-d+1)^{1/2}}(\det(AA^T))^{1/2+1/(2(n-d))}\,.
\label{general_via_A}
\end{equation}
\label{general_Frobenius}
\end{theorem}

In addition, the structural Theorem \ref{main2} allows us to obtain the lower and upper bounds for $\frob_k(A)$ in terms of a finite basis of the monomial ideal $I^k(A)$.
Let $||\cdot||_\infty$ denote the maximum norm.
\begin{theorem}
Let $x^{g_1}, \ldots, x^{g_t}$ be a finite basis for the ideal $I^k(A)$ and let $m(A)=\min_{1\le i\le t}||g_i||_{\infty}$. 
The number $\frob_k(A)$ satisfies the inequalities
\begin{equation}
m(A)-1\le \  \frob_k(A) \ \le \frac{n-d}{2(n-d+1)^{1/2}}\det(AA^T)+m(A)\left(\frac{{\det(AA^T)}}{n-d+1}\right)^{1/2}\,.
\label{boundsforgk}
\end{equation}
\label{ideal_bounds}
\end{theorem}


\item Third.
In Section \ref{kHoles}, we propose  a way to compute the $k$-holes, i.e., $\sg_{<k}(A)$, of the semigroup $\sg(A)$.
We give a natural generalization of the proof techniques used by Hemmecke et al.\cite{hemmecke}  that relies on Hilbert bases
to obtain the following theorem:

\begin{theorem}
\label{k-holes-hilb}
There exists an algorithm that computes for an integral matrix $A$ a
finite explicit representation for the set $\sg_{<k}(A)$ of $k$-holes of the
semigroup $\sg(A)$. The algorithm
computes (finitely many) vectors $h_i\in\Z^d$ and (affine) semigroups $M_i$, $i\in I$, each given
by a finite set of generators in $\Z^d$, such that
\begin{equation*}\label{General_Hemmecke}
\sg_{<k}(A)=\bigcup_{i\in I}\; \left(h_i+M_i\right).
\end{equation*}
\end{theorem}
Here $M_i$ could be trivial, that is, $M_i=\{0\}$. In the special case $k=1$, Theorem \ref{k-holes-hilb} was proved in  \cite{hemmecke}.

\item Fourth. While it is known that computing $k$-holes is NP-hard, even for the original Frobenius number case, here we show that
for $n$ and $k$ fixed integer numbers, there is
an efficient algorithm to detect all the $\geq k$-feasible vectors $b$'s, not explicitly one by one, but rather
the entire set of  $k$-feasible vectors is encoded as a single multivariate \emph{generating function}, $\sum_{\geq k-\text{feasible}} t^b$.

\begin{theorem} \label{main1} Let  $n$ and $k$ be fixed positive integers.
Let $A \in \Z^{d \times n}$ and let $M$ be a positive integer.  
Then there is a polynomial time algorithm to compute a short sum of
rational functions $G(t)$ which efficiently represents a formal sum 

$$\sum_{b \in \sg_{\geq k} (A),  \ b \in [-M,M]^d}  t^b.$$

Moreover, from the algebraic formula, one can perform the following tasks in polynomial time:

\begin{enumerate}
\item Count how many such $b$'s are there (finite because $M$ provides a box).
\item Extract the lexicographic-smallest such $b$, $\geq k$-feasible vector.
\item Find the $\geq k$-feasible vector $b$ that maximizes the dot product $c^T b$.
\end{enumerate}
\end{theorem}

Let us explain a bit the philosophy of such a theorem using generating functions for those not familiar with this point of view:
In 1993 A. Barvinok \cite{bar} gave an algorithm for counting the lattice points inside a polyhedron
$P$ in  polynomial time when the dimension of $P$ is a constant.
The input of the algorithm is the inequality description of $P$, the output is a
polynomial-size formula for the multivariate generating function of all lattice points in $P$,
namely $f(P)=\sum_{a \in P \cap \Z ^n} x^a$, where $x^a$ is an abbreviation of $x_1^{a_1} x_2^{a_2}\dots x_n^{a_n}$.
Hence, a long polynomial with exponentially many monomials  is encoded as a much shorter sum of rational functions
of the form

\begin{equation}
\label{barvinokseries}
f(P) \quad
= \quad
\sum_{i \in I} \pm \frac{x^{u_i}}{(1-x^{c_{1,i}})(1-x^{c_{2,i}})\dots
(1-x^{c_{s,i}})}.
\end{equation}

Later on Barvinok and Woods \cite{newbar} developed a set of powerful manipulation rules 
for using these short rational functions in Boolean constructions on various sets of lattice points, as well as a way to recover the
lattice points inside the image of a linear projection of a convex polytope. 


It must be remarked that from the results of Barvinok \cite{bar} for fixed $n$, but not necessarily fixed $k$,
one can decide whether a particular $b$ is  $k$-feasible in polynomial time. Recently Eisenbrand and H\"anhle \cite{eisenbrandhaenhle2013}
showed that the problem of finding the right-hand-side vector $b$ that maximizes the number of lattice points solutions, when $b$
is restricted to take values in a polyhedron, is NP-hard.  

%
%
%

One can prove a nice theorem for the computation of the $k$-Frobenius number  as  a corollary of Theorem \ref{main1}.
Recall that given a vector $(a_1,\dots, a_n)$ and a positive integer $b$ a \emph{knapsack problem} is a linear Diophantine problem of the form
$a_1x_1+\dots +a_nx_n=b$ with $x_i \geq 0$. The question about finding the $k$-Frobenius number is a query over a parametric family of
knapsack problems.

\begin{corollary} \label{k-frobeniuscomp} Let $n,k$ be two fixed positive integers.
Consider the parametric knapsack problem $a^T x=b$, $x\ge 0$ associated with the vector $a=(a_1,\dots,a_n)^T\in\Z_{>0}^n$ with
$\gcd(a_1,\dots,a_n)=1$. Then the $k$-Frobenius number $F_k(a)$ can be computed in polynomial time.
\end{corollary}

Corollary \ref{k-frobeniuscomp} greatly generalizes a similar celebrated theorem of R. Kannan \cite{Kannan} and extends
the proof of Barvinok and Woods for $k=1$. Section \ref{sectionBarvinokApplications} gives proofs of Theorem \ref{main1} and its Corollary \ref{k-frobeniuscomp}.

\item The fifth contribution of our article is concerned with practical computations of $=k$-feasible knapsack problems and $k$-Frobenius problems.
Given $a=(a_1,\ldots,a_n)\in \Z^n_{>0}$ with $\gcd(a_1, \ldots, a_n)=1$, let $g_k=g_k(a)$ denote the largest positive integer $b$
such that $a_1x_1+\cdots+a_nx_n=b$ has {\em exactly} $k$  nonnegative integral solutions
if such $b$ exists and zero otherwise. The numbers $g_k$ were studied by many authors, see e. g. \cite{BeckRobins:extension}, \cite{BeckKifer:extension}, \cite{brownetal} and \cite{shallit2011}.
In Section \ref{exactly_k} we prove the inequality $g_0<g_1$ for $n=3$. This answers in affirmative a question proposed in \cite{brownetal}.


\item
From Corollary \ref{k-frobeniuscomp} one further can ask: What is the computational complexity of computing the $k$-Frobenius number when the
dimension $n$ is fixed but $k$ is part of the input? On the basis of the results of  \cite{eisenbrandhaenhle2013} we suspect that this is an NP-hard problem,
but we do not know of the answer. It is then of interest to experiment with the values of $\frob_k (a)$ to see the growth for fixed values of $n$ and possibly predict
a formula. The sixth and final contribution of our article is about practical computation and experimental exploration on the behavior of the $k$-Frobenius numbers.
In Section \ref{kfrobcompute} we give an algorithm for fast practical computation of the $k$-Frobenius numbers 
using ideas from dynamic programming. Our experiments with knapsacks of three variables ($n=3$) provided support for a new conjecture on the asymptotic properties
of the average value of the $k$-Frobenius numbers.


\end{enumerate}

In what follows we assume that the reader is familiar with polyhedral convexity, monomial ideals, toric ideals, semigroup rings, linear Diophantine equations, and Gr\"obner
bases  as presented in \cite{newbar} and \cite{coxlittleoshea,sturmfels}.


\section{Proof of Theorem \ref{main2}: Monomial ideals and $\geq k$-feasibility}
\label{sectionMonomialIdeal}

%
%
%

For $f\in \cone(A)\cap \Z^d$ define
\begin{equation*}
L^k_{A,f}=\{\lambda\in \Z^n_{\ge 0}: IP_A(f+A\lambda) \mbox{ is }\ge k\mbox{ feasible}\}\,,
\end{equation*}
so that $ \sg_{\ge k}(A)=\{ A\lambda: \lambda \in L^k_{A,0}\} $.
Define then the monomial ideal $I^k(A)$ as follows (with its set of exponent vectors denoted by $E^k(A)$).
\begin{equation*}
I^k(A):=\langle x^\lambda: \lambda \in L^k_{A,0}\rangle\,.
\end{equation*}
To see that  the equation $\sg_{\ge k}(A)= \{A\lambda: \lambda \in E^k(A)\}$  is satisfied it is enough to check that
for any $\lambda_0\in L^k_{A,0}$ the inclusion
$
\lambda_0+\Z^n_{\ge 0}\subset L^k_{A,0}
$
holds. We will prove the following more general statement.
\begin{lemma}\label{staircase_inclusion}
For any $f \in \cone(A)\cap \Z^d$ and $\lambda_0\in L^k_{A,f}$ we have the inclusion
\begin{equation}
\lambda_0+\Z^n_{\ge 0}\subset L^k_{A,f}\,.
\label{shift}
\end{equation}
\end{lemma}

\begin{proof} Let $\lambda_0\in L^k_{A,f}$, so that there exist $k$ distinct vectors $\lambda_1,\ldots, \lambda_{k}\in \Z^n_{\ge 0}$ with
\begin{equation*}
 f+A\lambda_0=A\lambda_1=\cdots=A\lambda_{k}\,.
\end{equation*}
Take any vector $\mu \in \Z^n_{\ge 0}$ and set $\nu=\lambda_0+\mu$.
Then, clearly, we have
\begin{equation*}
f+A\nu=A(\lambda_1+\mu)=\cdots=A(\lambda_{k}+\mu)\,,
\end{equation*}
where all vectors $\lambda_1+\mu, \ldots, \lambda_{k}+\mu\in \Z^n_{\ge 0}$ are distinct.
Consequently, $IP_A(f+A\nu)$ is $\ge k$ feasible and, thus, $\nu\in L^k_{A,f}$.
Hence (\ref{shift}) holds and the lemma is proved.
\end{proof}

Lemma \ref{staircase_inclusion} with $f=0$ clearly implies the first claim of Theorem \ref{main2}.
Let us now prove the second claim.
Recall that the elements of the set $\sg_{<k}(A)$ are also called $k$-{\em holes}. A $k$-hole $f$ is {\em fundamental} if there is no other $k$-hole $h\in\sg_{<k}(A)$
such that $f-h\in \sg_{\geq 1}(A)$. 
In other words, for any $u\in \sg(A)$
$f-u\notin \sg_{<k}(A)$.



\begin{lemma}\label{zonotope}
The set of fundamental $k$-holes is a subset of the zonotope
\begin{equation*}
P=\left\{A\lambda: \lambda \in [0,1)^n\right\}\,.
\end{equation*}
\end{lemma}


\begin{proof} Let $f\in \sg_{<k}(A)$ be a fundamental hole. We can write
\begin{equation*}
f=A\lambda\,,\; \lambda\in \Q^n_{\ge 0}\,.
\end{equation*}
Suppose $f\notin P$. Then for some $j$ we must have $\lambda_j\ge 1$. Thus, denoting by $A_j$  the $j$th column vector of $A$,
the element $f'=f-A_j$ is a $k$-hole as any $k$ distinct solutions for $IP_A(f')$ would correspond to $k$ distinct solutions for $IP_A(f)$.
Thus we get a contradiction with our choice of $f$ as a fundamental $k$-hole. This implies $\lambda_j<1$ for all $j$ and, consequently,
$f\in P$. The lemma is proved.
\end{proof}

Lemma \ref{zonotope} shows, in particular, that the number of fundamental $k$-holes is finite.
Let us fix a fundamental $k$-hole $f$. If the set $L^k_{A,f}$ is empty then $f+A\lambda$ is a $k$-hole for all $\lambda\in \Z^n_{\ge 0}$. Assume now that $L^k_{A,f}$ is not empty and consider the monomial ideal $I^k_{A,f}\subset {\mathbb Q}[x_1,\ldots, x_n]$ defined as
\begin{equation*}
I^k_{A,f}=\langle x^\lambda: \lambda\in L^k_{A,f} \rangle\,.
\end{equation*}
Then, in view of (\ref{shift}), $f+A\lambda$ is not a $k$-hole if and only if $x^\lambda \in I^k_{A,f}$.

Thus we need to write down the set $C(I^k_{A,f})$ of exponents of {\em standard monomials} (the monomials not in the ideal) for  $I^k_{A,f}$.
Any such exponent $\lambda\in C(I^k_{A,f})$ corresponds to the $k$-hole $f+A\lambda$.

By Theorem 3 in Chapter 9 of \cite{coxlittleoshea}, the set $C(I^k_{A,f})$ can be written as a finite union of translates of coordinate subspaces of $\Z^n_{\ge 0}$.
Since the number of fundamental $k$-holes is finite, the second claim of Theorem \ref{main2} is proved.

\section{Asymptotic structure of $\sg_{\ge k}(A)$}
\label{Asymptotic}

In this section we assume that $A\in\Z^{d\times n}$, $1\leq d<n$, is an integral
$d\times n$ matrix satisfying
\begin{equation}
\begin{split}
{\rm i)}&\,\, \gcd\left(\det(A_{I_d}) : A_{I_d}\text{ is an $d\times
    d$ minor of }A\right)=1, \\
{\rm ii)}&\,\, \{x\in\R^n_{\ge 0}: A\,x=0\}=\{0\}\,.
\end{split}
\label{assumption}
\end{equation}
In the important
special case $d=1$ the matrix $A=a^T$ is just a row vector with $a=(a_1,\ldots, a_n)^T\in\Z^n$ and  \eqref{assumption} i) says that
 $\gcd(a_1,
\ldots, a_n)=1$. Due to the second assumption \eqref{assumption} ii)
we may assume that all entries of the vector $a$ are positive.
It follows that the largest integral value $b$ such that the problem $IP_A(b)$ is
$<k$-feasible, the $k$-Frobenius number $\frob_k(a)$, is well-defined. Clearly, for $d=1$ we have the inclusion
%
\begin{equation}
\interior(\frob_k(a)+\R_{\ge 0})\cap\Z\subset \sg_{\ge k}(a^T)\,,
\label{eq:frob_cone}
\end{equation}
where $\interior(\cdot)$ denotes the interior of the set.

In general, the structure of the set $\sg_{\ge k}(A)$, apart from a few special cases, is not well understood. It is known that, in analogy with (\ref{eq:frob_cone}), $\sg_{\ge k}(A)$ can be decomposed into the set of all integer points in the interior of a certain translated cone and a complex complementary set.
More recent results  (see \cite{AHL}) attempt to estimate the location of such a cone along the fixed direction  $v=A\,\1$, where $\1$ is the all-$1$-vector, in the interior of $\cone(A)$.
The choice of $v$ as the direction vector is dated back to the paper of Khovanskii \cite{Khovanskii1} for $k=1$.
In general, given any rational vector $u$ in the interior of $\cone(A)$, the multiple $t u$ will, for large enough $t\in \Z$, lie in  $\sg_{\ge k}(A)$. Hence the name {\em asymptotic structure}.

In this paper we consider a generalization of the Frobenius number that  reflects $\ge k$-feasibility
properties of the whole family of the problems $IP_A(b)$, when $b$ runs over all integer vectors in the interior of the cone $\cone(A)$.
Given a direction vector $b\in \interior(\cone(A))\cap\Z^d$ put
\begin{equation*}
\dfrob_k(A,b)=\min\{t\geq 0 : \interior(tb+\cone(A))\cap \Z^d \subset  \sg_{\ge k}(A)\}.
\end{equation*}
%
%
We define the $k$-{\em Frobenius number associated with} $A$ as
\begin{equation*}\begin{split}
\frob_k(A)=\max\{\dfrob_k(A,b): b\in \interior(\cone(A))\cap \Z^d\}\,.
%
\end{split}
\end{equation*}
%
%
%

The number $\dfrob_k(A)=\dfrob_k(A,v)$ was called in \cite{AHL} the {\em diagonal $k$-Frobenius number of} $A$.
A lower bound for $\dfrob_k(A)$ (and thus, by definition of $\dfrob_k(A)$, for $\frob_k(A)$) was given in \cite[Theorem 1.3]{AHL}.
Next we  derive the lower and upper bounds for $\frob_k(A)$ presented in Theorem \ref{general_Frobenius}  and Theorem \ref{ideal_bounds}.

Before we start the proofs it is worth remarking they will be based on the results obtained in \cite{AH}, \cite{AHL} and the structural Theorem \ref{main2}.

\section{Proof of Theorem \ref{general_Frobenius}}
\label{proof_general_Frobenius}

First we show that the $k$-Frobenius number $\frob_k(A)$ is bounded from above by the (suitably normalized) number $\dfrob_k(A)$.

\begin{lemma}
\begin{equation}
\frob_k(A)\le \left(\frac{{\det(AA^T)}}{n-d+1}\right)^{1/2}\,\dfrob_k(A)\,.
\label{general_via_diagonal}
\end{equation}
\label{lemma_general_via_diagonal}
\end{lemma}

\begin{proof}
Put for convenience $\gamma= \left(\frac{{\det(AA^T)}}{n-d+1}\right)^{1/2}$. As it was shown in the proof of Lemma 1.1 in \cite{AH}, for any $b\in \interior(\cone(A)) \cap\Z^d$ the vector $\gamma  b$
is contained in $v+\cone(A)$. Therefore $\gamma b+\cone(A)\subset v+\cone(A)$ and, consequently,
\begin{equation*}
\interior(\dfrob_k(A)\gamma b+\cone(A))\cap \Z^n \subset \interior(\dfrob_k(A)v+\cone(A))\cap \Z^n \subset \sg_{\ge k}(A)\,.
\end{equation*}
Hence for any $b\in \interior(\cone(A)) \cap\Z^d$ we have $\dfrob_k(A,b) \le \dfrob_k(A)\gamma$. Therefore $\frob_k(A)\le \dfrob_k(A)\gamma$ and the lemma is proved.
\end{proof}
The diagonal $k$-Frobenius number $\dfrob_k(A)$ in its turn is bounded from above in terms of $A$ and $k$ due to the following result.
\begin{theorem}[Theorem 1.2 in \cite{AHL}] The diagonal $k$-Frobenius number associated with $A$ satisfies the inequality
\begin{equation}
\dfrob_k(A)\leq
\frac{n-d}{2}(\det(AA^T))^{1/2} + \frac{(k-1)^{1/(n-d)}}{2}\,\left(\det(AA^T)\right)^{1/(2(n-d))}. \label{upper_bound_for_dFN}
\end{equation}
\label{upper_bound_AHL}
\end{theorem}
Combining (\ref{general_via_diagonal}) and (\ref{upper_bound_for_dFN}) we obtain the inequality (\ref{general_via_A}).

\section{Proof of Theorem \ref{ideal_bounds}}
\label{proof_ideal_bounds}


%
%

The set of exponents $E^k(A)$ of the monomial ideal $I^k(A)=\langle x^{g_1}, \ldots, x^{g_t} \rangle$
has the form
\begin{equation}
E^k(A)=\bigcup_{i=1}^t (g_i+\Z^n_{\ge 0})\,.
\label{unionofexponents}
\end{equation}
By (\ref{unionofexponents}), any $g\in E^k(A)$ has $||g||_{\infty}\ge m(A)$. Therefore the point $(m(A)-1)\1\notin E^k(A)$ and,  by Theorem \ref{main2} (i), we obtain $A((m(A)-1)\1)=(m(A)-1)v\notin \sg_{\ge k}(A)$. Therefore,
by the definition of $\dfrob_k(A)$,
\begin{equation*}
m(A)-1\le \dfrob_k(A)\le \frob_k(A)\,.
\end{equation*}
This proves the lower bound in (\ref{boundsforgk}).

To derive the upper bound, we will show first that $\dfrob_k(A)$ satisfies the inequality
\begin{equation}
\dfrob_k(A)\le \dfrob_1(A) + m(A)\,.
\label{boundsforgk_via_g1}
\end{equation}
Let us choose any $y\in \interior((\dfrob_1(A)+m(A))v+\cone(A))\cap\Z^d$. To prove (\ref{boundsforgk_via_g1}), it is enough to show that $y\in \sg_{\ge k}(A)$.
Consider the point $y'=y-m(A)v$. Since $y'\in  \interior(\dfrob_1(A)v+\cone(A))\cap\Z^d$, we have $y'\in \sg_{\ge 1}(A)$. Therefore, there exists $z\in \Z_{\ge 0}^n$ such that $Az=y'$.
Hence $A(m(A)\1+z)=m(A)v+Az=y$. Finally, observe that $m(A)\1\in E^k(A)$ by (\ref{unionofexponents}), and hence $m(A)\1+z\in m(A)\1+\Z_{\ge 0}^n\subset E^k(A)$. Consequently, $y=A(z+m(A)\1)\in \sg_{\ge k}(A)$ and the inequality (\ref{boundsforgk_via_g1})
is proved. The upper bound in (\ref{boundsforgk}) now follows from  (\ref{general_via_diagonal}), (\ref{boundsforgk_via_g1}) and (\ref{upper_bound_for_dFN}) with $k=1$.

\section{Proof of Theorem \ref{k-holes-hilb}: Computing $k$-holes via Hilbert bases}
\label{kHoles}

In this section we combine the results of Hemmecke et al. \cite{hemmecke} with our techniques
to compute the elements of $\sg_{<k}(A)$ proving Theorem \ref{k-holes-hilb}.
We present an
algorithm to compute an \emph{explicit} representation of $\sg_{< k}(A)$, even for an
infinite case, using semigroups. We remark that this explicit representation need not be of
polynomial size in the input size of $A$.

In view of the proof of Theorem \ref{main2} (ii), it is enough to
compute all fundamental $k$-holes and then for each fundamental $k$-hole $f$ with nonempty set $L^k_{A,f}$ compute the standard monomials of the ideal $I^k_{A,f}$.
If the set $L^k_{A,f}$ is empty, the set of all $k$-holes $H_{k}$  contains the translated semigroup $f+\sg(A)$.
By Lemma \ref{zonotope}, all fundamental $k$-holes are located in a zonotope $P=\left\{A\lambda: \lambda \in [0,1)^n\right\}$.
Thus, with a straightforward generalization of the approach proposed in Hemmecke et al. \cite{hemmecke}, the fundamental $k$-holes the can be computed by using a Hilbert basis of the cone $\cone(A)$.

%

Let $f$ be a fundamental $k$-hole. Recall that for a nonempty set $L^k_{A,f}$ the monomial ideal $I^k_{A,f}\subset {\mathbb Q}[x_1,\ldots, x_n]$ is defined as
\begin{equation*}
I^k_{A,f}=\langle x^\lambda: \lambda\in L^k_{A,f}\rangle\,
\end{equation*}
and $f+A\lambda$ is not a $k$-hole if and only if $x^\lambda \in I^k_{A,f}$.

Thus we need to compute the exponents of standard monomials for the ideal $I^k_{A,f}$. Any such exponent $\lambda\in \Z^n_{\ge 0}$ corresponds to the $k$-hole $f+A\lambda$.
The exponents of standard monomials can be computed explicitly from a set of generators of the ideal. Hence, it is enough to find the generators of $I^k_{A,f}$.
Let us fix an ordering $\prec$ in $\Z^n_{\ge 0}$.  The minimal generators for the ideal $I^k_{A,f}$ correspond to the $\prec$-minimal elements of the set
\begin{equation*}\begin{split}
L^k_{A,f}=\{\lambda\in \Z^n_{\ge 0}:  \exists\mbox{ distinct }\mu_1,\ldots,\mu_k\in  \Z^n_{\ge 0} \mbox{ such that }\\ f+A\lambda=A\mu_1=\cdots=A\mu_k\}\,.
\end{split}
\end{equation*}

For computational purposes it is enough to compute a set of vectors of $L^k_{A,f}$ that contains all the $\prec$-minimal elements. We will proceed as follows.
Let $K$ be a complete graph with the vertex set $V=\{1,2, \dots, k\}$.  By a weighted orientation $H$ of $K$ we will understand  a weighted directed graph  $H=(V,E)$
such that any two vertices of $H$ are connected by a directed edge $e\in E$ with a weight $w(e)\in\{1,\ldots, n\}$.

Let ${\mathcal S}$ be  set of all weighted orientations of $K$.
For each $H\in {\mathcal S}$ we construct the following two auxiliary sets:
the set
\begin{equation*}\begin{split}
L_{H}=\{\lambda\in \Z^n_{\ge 0}:  \exists\mu_1,\ldots,\mu_k\in  \Z^n_{\ge 0} \mbox{ such that }f+A\lambda=A\mu_1=\cdots=A\mu_k\\
\mbox{and } (\mu_{i})_{w(e)}\le (\mu_{j})_{w(e)}-1 \mbox{ for each } e=(i,j)\in E\,\}
\end{split}
\end{equation*}
and the set
\begin{equation*}\begin{split}
M_{H}=\{(\lambda, \mu_1, \ldots, \mu_k)\in \Z^{(k+1)n}_{\ge 0}:   f+A\lambda=A\mu_1=\cdots=A\mu_k\,\;\\ \mbox{ and } (\mu_{i})_{w(e)}\le (\mu_{j})_{w(e)}-1 \mbox{ for each } e=(i,j)\in E\}\,.
\end{split}\end{equation*}
Then, in particular, $L^k_{A,f}=\bigcup_{H\in {\mathcal S}} L_H$, where the union is taken over all orientations in $H\in {\mathcal S}$.

We will need the following result.

\begin{lemma}\label{proj}
Let $\lambda_0$ be a $\prec$-minimal element of $L_H$. Then there exists a $\prec$-minimal element of $M_{H}$ of the form $(\lambda_0, \hat \mu_1, \ldots, \hat \mu_k)$.
\end{lemma}

\begin{proof} Let $\lambda_0$ be a $\prec$-minimal element of $L_H$. Suppose on contrary, for every $(\mu_1, \ldots, \mu_k)\in \Z^{kn}_{\ge 0}$ the vector
$(\lambda_0,  \mu_1, \ldots,  \mu_k)$ is not a $\prec$-minimal element of $M_{H}$. Let $(\hat \mu_1, \ldots, \hat \mu_k)$ be a $\prec$-minimal element of
the set
\begin{equation*}\begin{split}
M_{H}|_{\lambda=\lambda_0}=\{(\mu_1, \ldots, \mu_k)\in \Z^{kn}_{\ge 0}:   f+A\lambda_0=A\mu_1=\cdots=A\mu_k\,\;\\ \mbox{ and } (\mu_{i})_{w(e)}\le (\mu_{j})_{w(e)}-1 \mbox{ for each } e=(i,j)\in E\}\,.
\end{split}\end{equation*}
By the assumption, there exists a vector $(\lambda',  \mu_1', \ldots,  \mu_k')\in  M_H$ such that $(\lambda',  \mu_1', \ldots,  \mu_k')\prec (\lambda_0,\hat \mu_1, \ldots, \hat \mu_k)$ and $(\lambda',  \mu_1', \ldots,  \mu_k')\neq (\lambda_0,\hat \mu_1, \ldots, \hat \mu_k)$.
If $\lambda'\neq\lambda_0$ we get a contradiction to the $\prec$-minimality of $\lambda_0$ in $L_H$. On the other hand, if $\lambda'=\lambda_0$ we get a contradiction to the $\prec$-minimality of $(\hat \mu_1, \ldots, \hat \mu_k)$ in $M_{H}|_{\lambda=\lambda_0}$.
\end{proof}

In view of Lemma \ref{proj}, to compute a generating set for $L^k_{A,f}$ (or to determine that $L^k_{A,f}$ is empty) it is now enough to compute the set of all minimal elements for $M_H, H\in {\mathcal S}$ and remove the last $kn$ components from each of them.


%
%
%
%
%
%
%
%

\section{Proof of Theorem \ref{main1}: Generating functions and $k$-feasibility}
\label{sectionBarvinokApplications}

We wish to prove a representation theorem of a set of lattice points as a sum $\sum_{ b \ \ge k \text{-feasible} \ b \in [-M,M]^d} t^b$. First we need a lemma about how to find an objective function that orders all points in a box.

\begin{lemma} Let $n$ be a constant denoting the number of variables. Given a positive integer $R$ and the associated $n$-dimensional box $B_R=[0,R]^n$, there exists an integer linear objective function $\bar{c}_R^\intercal x$ such that $\bar{c}_R^\intercal (y_i- y_j) \not= 0$  for all pairs of non-zero lattice points $y_i,y_j$ inside the box $B_R$. One can find one such vector in polynomial time.
\end{lemma}

\noindent {\em Proof:} Let $s$ be a single auxiliary (real) variable and the associated vector ${\bf c}(s)=(1,s,s^2,s^3,\dots,s^{n-1})^T$. Now for each of the $L=R^n (R^n-1)/2$ pairs of non-zero lattice vectors $y_i-y_j$ we construct one univariate polynomial $f_{i,j}(s)={\bf c}(s)^T (y_i-y_j)$ (since $y_i,y_j$ are distinct the polynomial $f_{i,j}$ is not identically zero). These are polynomials of degree $n-1$ so they can only have at most $n-1$ real roots each. Note also that these are polynomials all of whose coefficients are integer numbers between $-R$ and $R$, that means, by the famous Cauchy bound on the absolute values of roots of univariate polynomials that any of the real roots of any $f_{i,j}(s)$ must be bounded in above by $1+R$. Thus taking, for example, the value $s_0=R+2$ gives $\bar{c}_R={\bf c}(s_0)$ as an integer vector that totally orders all lattice points in the box $B_R$. Note that the bit-size description of $\bar{c}_R$ is polynomial in the input namely $n$, and $\log(R)$ because the entries are the first $n-1$ powers of $R$. The lemma is proved.

Below we will use the algorithmic technique of rational generating functions developed by Barvinok and Woods in \cite{bar,newbar} (see also the book \cite{barvibook}).
%

%
%

A key subroutine introduced by Barvinok and Woods is the following \emph {Projection Theorem}. 

\begin{lemma}[Theorem 1.7 in \cite{newbar}] \label{project}
Assume the dimension $n$ is a fixed constant.
Consider  a rational polytope $P \subset \real^n$ and a
linear map $T: \Z^n \rightarrow \Z^d$ such that $T(\Z^n)\subset \Z^d$.
There is a polynomial time algorithm which
computes a short representation of the
generating function $\, f \bigl(T(P \cap \Z^n),x\bigr) $.
\end{lemma}

Now all the set up of the proof of Theorem \ref{main1} is ready. 
Recall that $A$ is an integral $d \times n$ matrix and $n,k$ are constants. Let $M$ be a given positive integer as in the statement
of the theorem, and let $R$ be another integer positive number (to be set a bit later).
We can define the polyhedron (note $X_i$ denotes an $n$-dimensional vector so this polytope lives in $nk$-dimensional space):

\begin{equation*}
\begin{split}
Q(A,k,R,M)=\bigl\{ (X_1,X_2,\dots,X_k) \in {\real^{nk}} \ : \,\, A X_1= \cdots=A X_k, \, \bar{c}_R^\intercal X_i \geq \bar{c}_R^\intercal X_{i+1} +1, \\ \text{for} \ i=1,\dots, k-1 \,\,\hbox{and} \,  \,  R \geq X_1 \geq 0, \,   M \geq AX_1 \geq -M \bigr\}\,.
\end{split}
\end{equation*}

One can use Barvinok's algorithm to compute the generating function of the lattice points of $Q(A,k,R,M)$ \cite{barvibook}.
The polytope $Q(A,k,R,M)$ has the following two key properties: First, all its integer points represent \emph{distinct} $k$-tuples of 
integer points that are in some parametric polyhedron $P_A(b)=\{x : Ax=b, x\geq 0\}$. When we turn lattice points into monomials, $ z_1^{X_1} z_2^{X_2} \dots z_k^{X_k}$ 
has only those exponents where $X_i \not= X_j$. Namely, this is precisely the set of all  monomials coming from $k$-tuples of distinct vectors in $\Z^n_{\ge 0}$ 
that give the same value $A X_1=A X_2= \dots =A X_k$. Second, all vectors $X_1$ are in the box $B_R=[0,R]^n \subset \real^n$, and third, all vectors $AX_1$ are in the box $[-M,M]^d \subset \real^d$.
It is important to note if $R$ is too small the pre-images $X_1$, when projected by $A$, may not hit all points in $\sg_{\geq k} (A) \cap [-M,M]^d$ which is what we want. We will calculate a
sufficiently large $R$ below.  

%
%
%
%
%

We now apply a very simple linear map $T(X_1,X_2,\dots,X_k)=AX_1$, by multiplication with $A$.
This map yields of course for each $k$-tuple (which has $X_i \not= X_j$) the corresponding right-hand side vector $b=A X_1$
that has at least $k$-distinct solutions and $b \in [-M,M]^d$. From the generation function of the lattice points of $Q(A,k,R,M)$ we use Lemma \ref{project} to obtain a generating function expression 

$$f=\sum_{b \in  \text{projection of} \ Q(A,k,R,M): \, \text{with at least $k$-representations and} \ b \in [-M,M] ^d} t^b.$$

Now we must take care that this captures all such $b$ within the box $\{b :  -M\le b_i \leq  M\}$.  To achieve this we need to find a large enough value of $R$
that suffices to capture all representations of $b$'s within the box $B_R$ when projected.  Such $R$ can be calculated as the maximum value among $n$ linear programs, one
for each variable $x_i$, given by $ \max x_i  \ \text{subject to} \  -M \leq Ax \leq M, \ x_i \geq 0$.  With this choice of $R$, the generating function $f$ above gives all the desired values of $b$.
Which is the desired  short rational function which efficiently represents the sum $\sum_{\ge k \text{-feasible} \ b \in [-M,M]^d } t^b$. This proves the main
result in the body of the paper for $\ge k$-feasibility.  Now, because if one knows a description for  $\sg_{\ge k}(A)$  and $\sg_{\ge k+1}(A)$ 
one knows $\sg_{= k}(A)=\sg_{\ge k}(A) \backslash \sg_{\ge k+1}(A)$ and $\sg_{< k}(A)=\sg(A)\backslash \sg_{\ge k}(A)$.  The similar generating functions of the other
cases ($<k$, $=k$)  are just obtained as the difference of the generating functions.


Now that we have proved the main statement of Theorem \ref{main1}, we move to prove Parts (a) to (d) of the theorem.

\begin{enumerate}

\item[{\bf Part (a)}] If we have a generating function representation of
$$\sum_{b: \ b  \ \text{is} \ \ge k\text{-feasible}, \ b \in [-M,M]^d} t^b,$$
it has the form  $$f(t)=\sum_{i \in I} \alpha_i {t^{p_i} \over (1-t^{a_{i1}}) \cdots  (1-t^{a_{ik}})}. $$

Note that by specializing at $t=(1, \ldots, 1)$, we can count how many $b$'s are $\ge k$-feasible (again the set is finite because it fits inside a box). Remark  the substitution
is not immediate since $t=(1, \ldots, 1)$ is a pole of each fraction in the representation of $f$. This problem is solvable because it has been shown by Barvinok and Woods 
that this computation can be  handled  efficiently (see Theorem 2.6 in \cite{newbar} for details) and this proves Part (a).

\item[{\bf Part (b)}] This item is a direct corollary of the following extraction lemma.

\begin{lemma}[Lemma 8 in \cite{latte2} or Theorem 7.5.2 in \cite{AGTO}] \label{extractmonomial}
Assume the dimension $n$ is fixed.  Let $S \subset \Z^n_+$ be nonempty and finite set of lattice points.
Suppose the polynomial $\,f(S;z) = \sum_{\beta \in S} z^\beta
\, $ is represented as a short rational function and let $c$ be a cost
vector. We can extract the (unique) lexicographic largest leading
monomial from the set $\{x^{\alpha}: \alpha \cdot c = L, \, \alpha \in S\}$,
where $L := \max \{\alpha \cdot c: \alpha \in S\}$, in polynomial time.
\end{lemma}

\item[{\bf Part (c)}]  Barvinok and Woods developed a way to do monomial substitutions (not just $t_i=1$ as we used in Part (a)), 
where the variable  $t_i$ in the current series, is replaced by a new monomial $z_1^{a_1}z_2^{a_2}\cdots z_r^{a_r}$. Note that the
rational generating function $f = \sum_{b \in Q \cap\Z^d} b^{b}$ can give the evaluations of the $b$'s   for a given objective function
$c\in\Z^d$.  If we make the substitution $t_i=z^{c_i}$, the above equation yields a \emph{univariate} rational function in $z$:

\begin{equation}\label{Univariate Polynomial in z}
f(z)=\sum_{i\in I}{E_i\frac{z^{c\cdot u_i}}{\prod_{j=1}^d(1-z^{c\cdot v_{ij}})
}}.
\end{equation}

Moreover $f(z)=\sum_{b \in Q\cap\Z^d} z^{c\cdot b}$.  Thus we just need to find the (lexicographically) largest monomial in the sum in polynomial time. But this follows from Part (b).
\end{enumerate}

To conclude we see how to compute the $k$-Frobenius number efficiently when $n$ number of variables and $k$ are fixed.
\vskip 12pt

\noindent \emph{Proof of Corollary \ref{k-frobeniuscomp}}: We start by observing that there is an upper bound for the  $k$-Frobenius number. Indeed Theorem 1.1 in \cite{AFH} gives already an upper bound that is certainly smaller than $M= k(n-1)! a_1 a_2 \cdots a_n$. The $k$-Frobenius number must be smaller and thus we will use $M$ in the bounding box created in Theorem \ref{main1}.

Next we claim the same generating function descriptions obtained in Theorem \ref{main1} can be also obtained for the sets of those  $b$ which are $=k$-feasible, $\geq k$-feasible, or $<k$-feasible with the added condition that $b_i \leq M$. This is because
the generating functions of those sets of $b$'s  can be obtained from the set of $b$'s encoded by Theorem \ref{main1} through Boolean operations (intersection, unions, complements). Indeed we clearly have
$\sg_{\ge k+1}(A)\setminus \sg_{\ge k}(A)= \sg_{= k}(A)$ and  $\sg_{< k}(A)=\sg_{\ge k}(A)\setminus \sg_{= k}(A)$. The same identities hold under the intersection with the box
$[-M,M]^d$, thus the claim follows.

We may see now that Corollary \ref{k-frobeniuscomp} follows directly from what we achieved in Theorem \ref{main1} and the Boolean operation Lemma of Barvinok and Woods. Indeed, from Theorem \ref{main1} we have a rational function representation of the $k$-feasible $b$ for the knapsack problem (note that $R$ is calculated by LPs as in the proof of Theorem \ref{main1}).

$$f(t)=\sum_{i\in I}{E_i\frac{t^{c\cdot u_i}}{\prod_{j=1}^d(1-t^{c\cdot v_{ij}}) }}   =\sum_{b \in \text{projection of}  Q(A,k,R,M) \cap\Z^d, \, \geq k-\text{feasible} b \in [-M,M]^n} t^{c\cdot b}.$$

Clearly the $k$-Frobenius number is simply the largest (lexicographic) $b$, such that $t^b$ is {\bf not} in $f(t)$, it is in its complement. Note the choice of bound $M$ is such
that we indeed have the $k$-Frobenius number inside.  Then, for the complement $\overline{S}={\Bbb Z}_+ \setminus S$, we compute
the generating function $f(\overline{S}; x)=(1-t)^{-1}-f(t)$ and then we compute the largest such $t^b$ in the complement using Lemma \ref{extractmonomial}.

\section{Integers with exactly $k$ representations}
\label{exactly_k}

Let $a=(a_1,\ldots,a_n)\in \Z^n_{>0}$ with $\gcd(a_1, \ldots, a_n)=1$. Recall that $g_k(a)$ denotes the largest positive integer $b$ such that $a_1x_1+\cdots+a_nx_n=b$ has {\em exactly} $k$ integral nonnegative solutions
if such $b$ exists and zero otherwise.
In other words, a positive $g_k$ is the largest integer that has {\em exactly $k$ representations} by $a_1,\ldots,a_n$.
In the case $n=2$ Beck and Robins  \cite{BeckRobins:extension} obtained the formula
\begin{equation}\label{theoremn=3}
g_{k-1}(a_1,a_2)=ka_1a_2-a_1-a_2\,.
\end{equation}
Note that  $\frob_k (a_1, a_2)=g_{k-1}(a_1,a_2)$ and, in general, $\frob_k(a)=\max \{g_i(a): i \leq k-1 \}$.

Given a fixed vector $a$, the behavior of the sequence $\{g_i(a)\}_{i=0}^\infty$ is far from being simple for $n\ge 3$.
For instance, it has been observed in Brown et al.  \cite{brownetal}
that $g_i(a_1,a_2,a_3) $ is not necessarily increasing with $i$. For example,
$g_{14}(3,5,8)=52$ whereas $g_{15}(3,5,8)=51$ and therefore in that particular case  $g_{14}>g_{15}$.
Furthermore, Shallit and Stankewicz \cite{shallit2011} proved that for $i>0$
and $n=5$, the quantity $g_0-g_i$ is arbitrarily large and positive. They also give an
example of $g_0>g_1$ for $n=4$.

For $n=3$ Brown et al. \cite{brownetal} posed the question whether or not the inequality $g_0<g_1$ always holds.
In what follows we answer this question in affirmative.
\begin{theorem}\label{Browns_conj}
Given $a_1<a_2<a_3\in \mathbb Z_{>0}$,
we have
$$g_0(a_1,a_2,a_3) < g_1(a_1,a_2,a_3) . $$
\end{theorem}
To prove this theorem, we first consider the case in which all coefficients
are pairwise relatively prime.
\begin{lemma}
Let $a_1<a_2<a_3$ be the pairwise relatively prime positive integers.
 Then
$$g_0(a_1,a_2,a_3) < g_1(a_1,a_2,a_3). $$
\label{relprime}
\end{lemma}
\begin{proof}
Let $f=g_0(a_1,a_2,a_3)$ denote the Frobenius number of the vector $(a_1,a_2,a_3)$.
Then
\begin{equation}
a_1x_1+a_2x_2+a_3x_3=f
\label{eq f}
\end{equation}
has no nonnegative integral solution.
On the other hand
\begin{equation}
a_1x_1+a_2x_2+a_3x_3=f+a_1
\label{eqfa1}
\end{equation}
has at least one nonnegative integral solution.
We will prove that this solution is unique. 
Let $\bar x$ be any nonnegative integral solution to \eqref{eqfa1}.
Observe that $\bar x_1=0$, otherwise we could trivially construct a nonnegative integral
solution for \eqref{eq f}. Therefore $\bar x=(0,\bar x_2,\bar x_3).$
Consider by contradiction that $\tilde x=(0,\tilde x_2,\tilde x_3)$ is another
nonnegative integral solution to \eqref{eqfa1}. Then $a_2(\bar x_2-\tilde x_2)+ a_3(\bar x_3-\tilde x_3)=0.$
In particular, since $\text{gcd}(a_2,a_3)=1$, this implies $|\bar x_2-\tilde x_2|= l {a_3}$
with $l$ a natural number,
which in turn implies that either $\bar x_2\geq a_3$ or $\tilde x_2\geq a_3$.
Assume w.l.o.g. that $\bar x_2\geq a_3$, and since $(0,\bar x_2,\bar x_3)$ is a nonnegative integral solution
to \eqref{eqfa1}, this implies that $f+a_1\geq a_2a_3$ and hence $f\geq a_2a_3-a_1$
which is in contradiction with the fact that $f=g_0(a_1,a_2,a_3)\leq g_0(a_2,a_3)
\leq  g_0(a_2,a_3)=a_2a_3-a_2-a_3.$\bigskip
\end{proof}
The following result was proved in \cite{brownetal}.
\begin{lemma}[Theorem 1 in Brown et al. \cite{brownetal} with $k=3$]
Let $d=\gcd(a_2,a_3)$ and $j\in \mathbb Z_{\ge 0}$ then either
$g_j(a_1,a_2,a_3)=dg_j(a_1,a_2/d,a_3/d)+(d-1)a_1$ or $g_j(a_1,a_2,a_3)=g_j(a_1,a_2/d,a_3/d)=0$, i.e.,
no right-hand-side achieves exactly $j$ integral solutions.
\label{getridofgcd}
\end{lemma}
To prove Theorem \ref{Browns_conj} we will reduce the general case to the case of pairwise
relatively prime coefficients.
Consider a triple $(a_1,a_2,a_3)$. If the coefficients are pairwise relatively prime,
the result follows from Lemma \ref{relprime}.
Assume that there is some g.c.d. different from 1 for a pair of coefficients.
From Lemma \ref{getridofgcd}, we can get rid of the g.c.d. and that does not change the
relative order between $g_0$ and $g_1$. By applying at most three times
Lemma \ref{getridofgcd}, we come back to the case where all coefficients are pairwise
relatively prime and the result follows.
The theorem is proved.


\section{Computing $F_k(a)$ by dynamic programming and the behavior of $F_k(a_1,a_2,a_3)$} \label{kfrobcompute}

 The computational aspects of the classical Frobenius problem
were studied by many authors (see \cite{fasterfrob} and references therein).
In this section, we propose a practical dynamic programming approach that allows us to compute the number of integral solutions to the knapsack problems and the $k$-Frobenius numbers.

\subsection{A simple dynamic programming algorithm for $F_k(a)$}\label{sec: algorithm}

Given $a_1,\ldots, a_n\in \mathbb Z_{>0}$, we denote by
$T_i(b)$ the number of integral solutions  of the
knapsack problem  $\sum_{l=1}^n a_l x_l=b, x\in \Z^n_{\ge 0}$, satisfying $x_j=0$ for $j<i$ and $x_i\geq 1,$
i. e., $T_i(b)$ counts the number of integral solutions of the knapsack problem where the smallest
nonzero index is $i$.
The idea of the algorithm is to update an array $T_i(b)$ for increasing $b$ and for all $i.$
The following observation allows us to initialize the dynamic programming
approach.
\begin{lemma}
$T_i(a_i)=1$ and $T_i(b)=0$ for all $0\leq b\leq a_i-1.$
\label{initdp}
\end{lemma}
The following lemma explains how to update the function $T$.
\begin{lemma}
Given $b\neq a_i$,
\begin{align}
T_i(b) = \sum_{j=i}^n T_j(b-a_i).
\label{dp relation}
\end{align}
\label{maindp}
\end{lemma}
\begin{proof}
We first prove that \eqref{dp relation} holds with $\geq$ instead of $=$.
Indeed consider  any nonnegative integral solution $\bar x$ to $\sum_{l=1}^n a_l x_l = b-a_i$
with $\bar x_j=0$ for all $1\leq j\leq i-1$, it can be transformed
into a nonnegative integral solution for $\sum_{l=1}^n a_l x_l = b$ by considering $\bar x+e_i$.
Obviously the first $(i-1)$ components are still zero and the $i^{th}$ component
is positive.

We now prove that \eqref{dp relation} holds with $\leq$.
Consider a solution $\hat x\in \mathbb Z_{\ge 0}^n$ with $\hat x_j=0 $ for all $1\leq j\leq i-1$
and $\hat x_i\geq 1$ to $\sum_{l=1}^n a_l x_l = b$. By subtracting $1$ from the $i$th component, we obtain
a nonnegative integer solution to $\sum_{l=1}^n a_l x_l = b-a_i$ where the first $i-1$ components are zero.
\end{proof}

Using Lemma \ref{initdp} and Lemma \ref{maindp}, we can fill in an array $T_i(b)$
starting from $b=0$
for increasing values of $b$.
Obtaining the array $T$ allows us to count the number of different solutions.
\begin{lemma}
The number of nonnegative integral solutions to $\sum_{l=1}^n a_l x_l =b $ is equal to
$\sum_{i=1}^n T_i(b).$
\end{lemma}
\begin{proof}
This follows from the fact that any integral solution to $\sum_{l=1}^n a_l x_l =b$
is counted exactly in one set corresponding to the smallest index for which $x_i$ is nonzero.
\end{proof}

 To determine the $k$-Frobenius number, we need to know
the largest $b$ that is $<k$-feasible. It is therefore important to determine a stopping
criterion for the dynamic programming algorithm. A first obvious criterion is to use
the upper bound for $b$ given in \cite{AFH}. It turns out that this upper bound is very often
too large compared to the actual $k$-Frobenius number and leads to longer
computation times. The following lemma
allows us to interrupt the computation with the guarantee for a $k$-Frobenius number earlier.
\begin{lemma}
Assume that $a_1<a_2<\cdots <a_n$ and that
$\sum_{i=1}^n a_ix_i=b$  has at least $k$ integral solutions for all $\bar b\leq b\leq \bar b+a_1$.
Then $\sum_{i=1}^n a_i x_i = b$ has at least $k$ integral solutions for all $b\geq \bar b$.
\label{stopcri}
\end{lemma}
\begin{proof}
This follows from the fact that $T_1(b) = \sum_{i=1}^n T_i(b-a_1)$ and that the number of integral solutions
of $\sum_{i=1}^n a_ix_i=b$ is at least $T_1(b)$.
\end{proof}

We now have all the necessary ingredients to describe a dynamic programming-based algorithm.

\begin{algorithm}
\caption{DP algorithm for the $k$-Frobenius number}
\begin{algorithmic}
\REQUIRE $a_1<a_2\cdots < a_n$, $k\geq 1$
\STATE $T_i(a_i)=1$ for all $i=1,\ldots, n$
\STATE $T_i(b)=0$ for $i=1,\ldots, n$ and $b<a_i$
\STATE  $b:=a_1+1$
\WHILE{$\exists \bar b\in [b-a_1,b-1]$ with $\sum_{i=1}^n T_i(\bar b) < k$}
\FOR{$i:=1 $ to $n$}
\STATE $T_i(b):= \sum_{j=i}^n T_j(b-a_i)$
\ENDFOR
\STATE  $b:=b+1$
\ENDWHILE
\STATE Return the largest $b$ that has less than $k$ solutions
\end{algorithmic}\label{algodp}
\end{algorithm}

We have done a number of computational experiments in order to evaluate the empirical complexity of Algorithm \ref{algodp}.
Our implementation of the algorithm was in standard C code with CPU Intel Core i7-975, 4 cores, 8 threads, CPU clock 3.33 GHz.
In the experiments we checked the average time needed to compute the $k$-Frobenius number for some common values
of $n$ and $a_1$, as well as the dependency on $k$. The algorithm demonstrated surprisingly good performance.
For instance, to check the dependency on $k$, we fixed the dimension to $n=10$ and
the range of values for $a_i$ to $[2^{14},2^{16}]$. For each value of $k$, we run 20 experiments.
The summary of the results is reported in Table \ref{kexperiment}.
\begin{table}
\begin{center}
\begin{tabular}{|l|ccccccccc|}
\hline
\multicolumn{10}{|c|}{Dimension $n=10$}\\
\hline
$k$ & 1 &  4 & 16 & 64 & 256 & 1024 & $2^{12}$ & $2^{14}$ & $2^{16}$ \\
\hline
Average time (s) & 0 &  0 & 0 & 0 & 0 & 0 & 1 & 1 & 1 \\
\hline
\end{tabular}
\end{center}
\caption{Dependency on $k$}
\label{kexperiment}
\end{table}
We observe that the average time grows very slowly with $k$.
 We utilize this property of Algorithm \ref{algodp} in the next section.

\subsection{Experiments on the $k$-Frobenius number $\frob_k(a_1,a_2, a_3)$}
\label{Frobenius_comp}

In this section we report the results of several experiments concerning the average behavior of the $k$-Frobenius number $\frob_k(a)$ and
make a conjecture on the basis of these results.
The experiments were performed using the dynamic programming algorithm introduced in Section \ref{sec: algorithm}.

Let $G_T=\{a=(a_1, a_2, a_3)\in \Z^3_{>0}: \gcd(a_1, a_2, a_3)=1\,, ||a||_\infty\le T\}$ and let $A_k(T)$ denote the average value of the (normalized) $k$-Frobenius number over $G_T$, that is
\begin{equation}
A_k(T)=\frac{1}{|G_T|}\sum_{(a_1, a_2, a_3)\in G_T}\frac{F_k(a_1, a_2, a_3)}{\sqrt{a_1 a_2 a_3}}\,.
\end{equation}
It follows from the results of Ustinov \cite{Ustinov_sbornik} that
\begin{equation}
A_1(T)\rightarrow \frac{8}{\pi}\;\;\; \mbox{as}\;\; T\rightarrow \infty\,.
\label{Ustinov_formula}
\end{equation}
%
%
For the case $k>1$ the upper bound in Theorem 1.1 of \cite{AFH}, together with (\ref{Ustinov_formula}) implies that for any $\epsilon>0$ for large enough $T$
\begin{equation*}
A_k(T) \le \frac{8}{\pi} + \sqrt{2(k-1)}+\epsilon \,.
\end{equation*}
On the other hand, from the lower bound  in Theorem 1.1 of \cite{AFH} and the proof of Proposition 1 in \cite{AHH} we conclude that for large enough $T$
\begin{equation}
A_k(T) >  \sqrt{2k}-\epsilon \,.
\label{lower_asympt}
\end{equation}

In our experiments we empirically obtained the values for $A_k(T_0)$ for fixed $T_0$ and various values of $k$. For instance, for $T_0=3000$ we have drawn 2000 triples of coefficients in the range $[2,T_0].$
For each triple, we have computed the $k$-Frobenius number for $k=10000,30000,60000,100000,150000.$
In Figure \ref{figevol}, we represent, with a logarithmic scale, the product $a_1a_2a_3$
on the $x$-axis and the $k$-Frobenius on the $y$-axis for the respective values of $k$.
\begin{figure}
\begin{center}
\includegraphics[width=.5\linewidth]{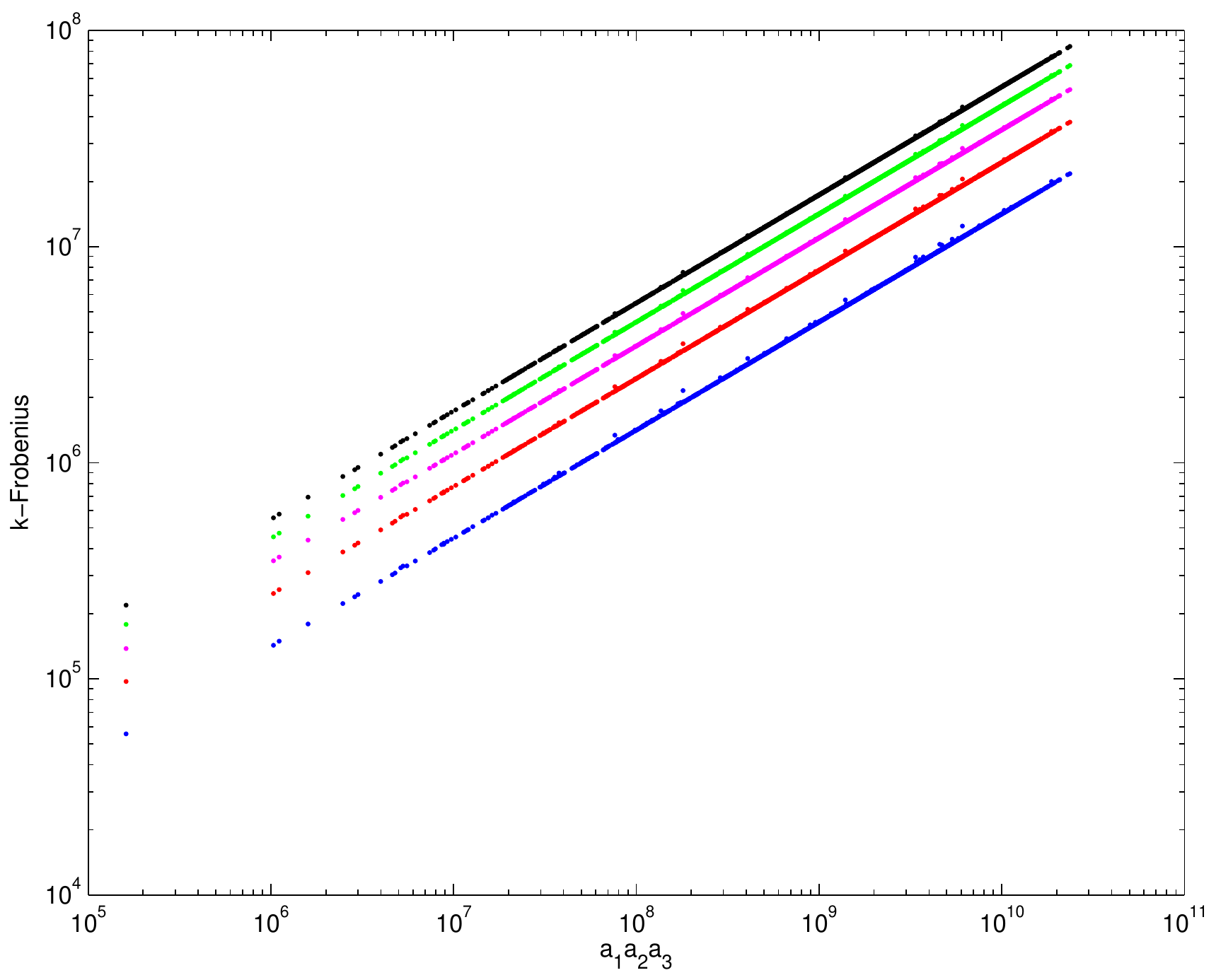}
\caption{Value of the $k$-Frobenius (over the $y$-axis) with respect to the product
$a_1a_2a_3$ for $a_i\in [2,3000]$ for $k=10000,30000,60000,100000,150000.$}
\label{figevol}
\end{center}
\end{figure}
In order to assess an approximate value for $A_k(T_0)$
we now perform a linear regression over $\log (a_1a_2a_3)$ and $\log (F_k(a_1,a_2,a_3))$
whose results are shown in the following table.
\begin{center}
\begin{tabular}{|l|c|c|}
\hline
$k$ & $A_k(T_0)$ & $\sqrt{2k}$\\ 
\hline
\;10000\; & \;141.5535\; & \;141.4214\;\\
\;30000\; & \;245.0937\;&  \;244.9490\;\\
\;60000\; & \;346.5570\; & \;346.4102\;\\
\;100000\; & \;447.3675\; & \;447.2136\;\\
\;150000\; & \;547.8815\; & \;547.7226\;\\
\hline
\end{tabular}
\end{center}

The obtained computational results allow us to make the following conjecture: 
\vskip .3cm

\noindent {\bf Conjecture:} Let $G_T=\{a=(a_1, a_2, a_3)\in \Z^3_{>0}: \gcd(a_1, a_2, a_3)=1\,, ||a||_\infty\le T\}$ and let $A_k(T)$ denote the average value of the (normalized) $k$-Frobenius number over $G_T$. Then

\begin{equation*}
\limsup_{k} |A_k(T)- \sqrt{2k}| \rightarrow 0\;\;\; \mbox{as}\;\; T\rightarrow \infty\,,
\end{equation*}
that is $A_k(T)$ asymptotically approaches the term $\sqrt{2k}$ of its lower bound (\ref{lower_asympt}).


\section{Acknowledgements} We are truly grateful to Lenny Fukshansky, Martin Henk, and Matthias K\"oppe
for their comments, encouragement and corrections that improve the paper we present today. The second author is 
grateful for partial support through an NSA grant. Some of our results were first announced  in \cite{ouripco}.

\end{document}